\documentclass{birkjour}

\usepackage[colorlinks]{hyperref}
\usepackage{dsfont}
\usepackage{amssymb}
\usepackage{mathrsfs}
\usepackage{dsfont}
\usepackage{bm}
\usepackage{enumerate}

\newtheorem{thm}{Theorem}[section]
\newtheorem{cor}[thm]{Corollary}
\newtheorem{lem}[thm]{Lemma}
\newtheorem{prop}[thm]{Proposition}

\theoremstyle{definition}
\newtheorem{defn}[thm]{Definition}

\theoremstyle{remark}
\newtheorem{rem}[thm]{Remark}
\newtheorem{ex}[thm]{Example}

\numberwithin{equation}{section}

\usepackage{xspace}
\newcommand{\wcbs}{\textsc{wcbs}\xspace}

\newcommand{\WC}{\mathbb{W}}

\renewcommand{\H}{\mathcal{H}}
\newcommand{\B}{\mathbb{B}}
\newcommand{\K}{\mathbb{K}}

\newcommand{\T}{\bT}
\newcommand{\F}{\Lambda}
\newcommand{\G}{\Gamma}

\newcommand{\W}{\mathcal{W}}
\newcommand{\diff}{\,\mathrm{d}}

\newcommand{\FI}{$\Rightarrow$}
\newcommand{\IFF}{$\Leftrightarrow$}

\newcommand{\ol}{\overline}

\newcommand{\e}{\epsilon}

\newcommand{\w}{\omega}
\renewcommand{\c}{\mathscr{C}}
\newcommand{\scrA}{\mathscr{A}}
\newcommand{\M}{\mathscr{M}}
\newcommand{\C}{\mathbb{C}}

\newcommand{\N}{\mathbb{N}}

\newcommand{\co}{\operatorname{co}}
\newcommand{\cco}{\operatorname{\overline{\co}}}

\newcommand{\A}{A}
\newcommand{\E}{E}%\fA}
\newcommand{\I}{\mathrm{i}}

\newcommand{\bT}{{\bm T}}

\newcommand{\fA}{E}%\mathfrak{A}}

\newcommand{\cA}{\mathcal{A}}

\newcommand{\cF}{\mathcal{F}}

\newcommand{\cG}{\mathcal{G}}

\newcommand{\cX}{\mathcal{X}}

\renewcommand{\a}{\mathbf{a}}

\newcommand{\bx}{\mathbf{x}}

\newcommand{\al}{\alpha}

\newcommand{\be}{\beta}

\newcommand{\la}{\lambda}

\newcommand{\de}{\delta}

\newcommand{\Om}{\Omega}

\newcommand{\set}[1]{\{#1\}}
\newcommand{\bigset}[1]{\bigl\{ #1 \bigr\}}
\newcommand{\Bigset}[1]{\Bigl\{ #1 \Bigr\}}

\newcommand{\bigabs}[1]{\bigl\lvert #1 \bigr\rvert}
\newcommand{\Bigabs}[1]{\Bigl\lvert #1 \Bigr\rvert}
\newcommand{\biggabs}[1]{\biggl\lvert #1 \biggr\rvert}

\newcommand{\bigprn}[1]{\bigl( #1 \bigr)}
\newcommand{\Bigprn}[1]{\Bigl( #1 \Bigr)}
\newcommand{\biggprn}[1]{\biggl( #1 \biggr)}

\newcommand{\enorm}{\lVert\,\cdot\,\rVert}

\newcommand{\Bignorm}[1]{\Bigl\lVert #1 \Bigr\rVert}

\newcommand{\lVertt}{\lvert\mspace{-2mu}\lvert\mspace{-2mu}\lvert}
\newcommand{\rVertt}{\rvert\mspace{-2mu}\rvert\mspace{-2mu}\rvert}
\newcommand{\enormm}{\lVertt\, \cdot \,\rVertt}
\newcommand{\normm}[1]{\lVertt #1 \rVertt}

\newcommand{\ip}[1]{\langle #1 \rangle}

\renewcommand{\leq}{\leqslant}
\renewcommand{\geq}{\geqslant}

\begin{document}

%-------------------------------------------------------------------------
% editorial commands: to be inserted by the editorial office
%
%\firstpage{1} \volume{228} \Copyrightyear{2004} \DOI{003-0001}
%
%\seriesextra{Just an add-on}
%\seriesextraline{This is the Concrete Title of this Book\br H.E. R and S.T.C. W, Eds.}
%
% for journals:
%
%\firstpage{1}
%\issuenumber{1}
%\Volumeandyear{1 (2004)}
%\Copyrightyear{2004}
%\DOI{003-xxxx-y}
%\Signet
%\commby{inhouse}
%\submitted{March 14, 2003}
%\received{March 16, 2000}
%\revised{June 1, 2000}
%\accepted{July 22, 2000}
%---------------------------------------------------------------------------

\title[A Characterization of compact operators]%
{A characterization of compact operators\\ on $\ell^p$-spaces}

\author[M. Abtahi]{Mortaza Abtahi}
\address{School of Mathematics and Computer Sciences\\
Damghan University\\
Damghan, P.O.BOX 36715-364\\
Iran}

\email{abtahi@du.ac.ir; mortaza.abtahi@gmail.com}



%\thanks{This work was completed with the support of }

\date{\today}

%\dedicatory{aaaaaaa}

\begin{abstract}
  Let $A$ be a Banach space, $p>1$, and $1/p+1/q=1$. If a sequence $\a=(a_i)$ in $A$ has a finite $p$-sum,
  then the operator $\F_\a:\ell^q\to A$, defined by $\F_\a(\be)=\sum_{i=1}^\infty \be_i a_i$, $\be=(\be_i)\in \ell^q$,
  is compact. We present a characterization of compact operators $\F:\ell^q\to A$, and
  prove that $\F$ is compact if and only if $\F=\F_\a$, for some sequence $\a=(a_i)$ in $A$
  with $\bigset{\bigprn{\phi(a_i)}: \phi\in A^*, \|\phi\|\leq 1}$ being
  a totally bounded set in $\ell^p$. For a sequence $(T_i)$ of bounded operators on
  a Hilbert space $\H$, the corresponding operator $\T:\ell^q\to \B(\H)$, defined by
  $\T(\be) = \sum_{i=1}^\infty \be_i T_i$, is compact if and only if the %joint numerical range
  set $\set{\ip{\T x,x}:\|x\|=1}$ is a totally bounded subset of $\ell^p$, where
  $\ip{\T x,x} = (\ip{T_1 x,x}, \ip{T_2 x,x}, \dotsc)$, for $x\in \H$.
  Similar results are established for $p=1$ and $p=\infty$.
\end{abstract}

\maketitle
%\tableofcontents

\section{Introduction}
\label{sec:intro}

Compact operators play a crucial role in functional analysis and operator theory, particularly in the study
of infinite-dimensional spaces. Their significance lies in their well-behaved spectral properties,
approximation capabilities, and their utility in solving complex problems in mathematical physics, differential equations,
and numerical analysis. 
%Exhibiting properties similar to matrices in finite-dimensional spaces, compact operators
%provide a powerful tool for simplifying and approximating the behavior of more general operators. 
Researchers have
long been drawn to their study, not only for their theoretical elegance but also for their practical applications.
Investigations into their geometric properties, spectral behavior, and interactions with other classes of operators
have established compact operators as a central topic in modern mathematical research. For further exploration
of their geometric properties, we recommend \cite{Bottazzi-approx-com-oper-2013,
Mal-Sain-geo-prop-oper-2019,Sain-approx-comp-oper-2021,
Wojcik-Ortho-compact-oper-2017} and references therein.

Among the first examples of Banach spaces to be systematically studied and have often served as the inspiration
for many concepts in Banach space theory, are the $\ell^p$ spaces. This paper deals with the characterization
of compact operators from $\ell^q$, where $1/p+1/q=1$, to an arbitrary Banach space $A$. It is proved that 
a bounded operator $\F:\ell^q\to A$ is compact if and only if it is a diagonal map
$\F\be = \sum_{i=1}^\infty \be_i a_i$, induced by a sequence $\a=(a_i)$ in $A$ whose 
dual shadow $\set{(\phi(a_i)): \phi\in A^*, \|\phi\|=1}$ is totally bounded in $\ell^p$.

Compact sets in $\ell^p$ play a crucial role in this work. A necessary and sufficient condition for a subset of $\ell^p$
to be compact is given in the following result, known as Kolmogorov compactness theorem;
see \cite[Theorem 4]{Kolmogorov-compactness-thm}.

\begin{thm}
\label{thm:Kolmogorov-Riesz-compactness}
  A set $K$ in $\ell^p$, where $p\in [1,\infty)$, is totally bounded if, and only if,
  $K$ is pointwise bounded and, for every $\e > 0$, there is $m\in\N$ such that
  \begin{equation}\label{eqn:Kolmogorov-Riesz-compactness}
    \Bigprn{\sum_{i > m} |\al_i|^p}^{1/p} < \e,\
    (\al_1,\al_2,\dotsc) \in K.
  \end{equation}
\end{thm}

For $1\leq p < \infty$, the dual space $(\ell^p)^*$ is isometrically isomorphic to $\ell^q$;
every $\psi\in (\ell^p)^*$ can be identified with a unique
$\be\in \ell^q$ such that $\psi=\hat\be$, where
\begin{equation*}
  \hat\be (\al) = \sum_{i=1}^\infty \al_i \be_i \quad (\al\in \ell^p).
\end{equation*}

%Moreover, the operator norm of $\hat\be$ equals $\|\be\|_q$; that is,
%\begin{equation}\label{eqn:norm-beta}
%  \|\hat \be\| =
%   \sup\Bigset{\Bigabs{\sum_{i=1}^\infty \be_i\al_i}:\|\al\|_p\leq 1}
%   = \|\be\|_q.
%\end{equation}

In particular, $\ell^p$ is reflexive, when $1<p<\infty$. Note that $\ell^1$ 
is isometrically isomorphic to the dual space $(c_0)^*$, where $c_0$ 
is the closed subspace of $\ell^\infty$ consisting of sequences that converge to $0$.

\subsection*{Notation}
Throughout this paper, vector spaces are assumed to be over the complex field $\C$.
Given a Banach space $\fA$, the closed unit ball of $\fA$ is denoted by $\fA_1$, that is,
$\fA_1=\set{u\in E: \|u\| \leq 1}$. We let $\fA^*$ denote the topological dual of $\fA$.
For every $u\in \fA$, define $\hat u:\fA^*\to \C$ by $\hat u(\psi)=\psi(u)$, for $\psi\in \fA^*$. The map
$u\mapsto \hat u$ embeds $\fA$ as a closed subspace of the second dual $\fA^{**}$.
If this map is surjective, then $\fA$ is called reflexive. Given another Banach space $A$, let $\B(\fA,A)$
represent the space of bounded operators of $\fA$ to $A$. An operator $\F:\fA\to A$ is compact if $\F(\fA_1)$
is relatively compact in $A$. The subspace of compact operators is denoted by $\K(\fA,A)$.
We abbreviate $\B(\fA,\fA)$ and $\K(\fA,\fA)$ to $\B(\fA)$ and $\K(\fA)$, respectively.

\subsection*{Outline}
In Section \ref{sec:wcbs}, for Banach spaces $\E$ and $A$,
we briefly investigate basic properties of the linear operators
$\F:\fA^*\to A$ that are weak*\,continuous on bounded sets (\wcbs, in short). We show that
\wcbs operators are compact, and that the converse holds only if $\fA$ is reflexive.
Main results on the characterization of compact operators are presented in Section \ref{sec:compact}.
For $1<p\leq \infty$ and $1/p+1/q=1$, we prove that $\F:\ell^q\to A$ is compact if and only if
there exists a sequence $\a=(a_1,a_2,\dotsc)$ in $A$ such that
\begin{enumerate}[(1)]
  \item  the set $\set{(\phi(a_1),\phi(a_2),\dotsc):\phi\in A^*_1}$ is totally bounded in $\ell^p$,
  \item $\F(\be)=\sum_{i=1}^\infty \be_ia_i$, for $\be=(\be_1,\be_2,\dotsc)$ in $\ell^q$.
\end{enumerate}

In Section \ref{sec:C(Omega)}, we specifically discuss the case of $A=\c(\Om)$, where
$\Om$ is a compact Housdorff space. We show that, for a sequence $F=(f_1,f_2,\dotsc)$ in $\c(\Om)$,
the following are equivalent;
\begin{enumerate}[(1)]
  \item the operator $\F_F:\ell^q\to \c(\Om)$, $\F_F(\be) = \sum_{i=1}^\infty \be_i f_i$, is compact,
  \item the function $F:\Om\to\ell^p$, $F(s)= (f_1(s),f_2(s),\dotsc)$, is continuous,
  \item the image set $F(\Om)=\set{F(s):s\in\Om}$ is totally bounded in $\ell^p$.
\end{enumerate}

In Section \ref{sec:B(H)}, we discuss the special case of $\cA=\B(\H)$,
where $\H$ is a Hilbert space. We prove that, for a sequence $\T=(T_1,T_2,\dotsc)$ in $\B(\H)$,
the following are equivalent;
\begin{enumerate}[(1)]
  \item the operator $\T:\ell^q\to\B(\H)$, $\T\be = \sum_{i=1}^\infty \be_i T_i$, is compact,
  \item the set $\set{\ip{\T x,y}:\|x\|=\|y\|=1}$ is totally bounded in $\ell^p$,
  \item the set $\set{\ip{\T x,x}:\|x\|=1}$ is totally bounded in $\ell^p$,
\end{enumerate}
where
  \begin{equation*}
   \ip{\T x,y} = (\ip{T_1x,y},\ip{T_2x,y},\dotsc) \quad (x,y\in \H).
  \end{equation*}

\section{Compactness and weak*\,continuity}
\label{sec:wcbs}

Let $\fA$ and $\A$ be Banach spaces.
  A linear operator $\F:\fA^*\to \A$ is called \emph{weak* continuous on bounded sets}
  (\wcbs, in short) if, for every bounded net $(\psi_\la)$ in $\fA^*$, the condition $\psi_\la\to \psi$
  in the weak*\,topology, implies that $\F(\psi_\la) \to \F(\psi)$ in $\A$.
  The space of \wcbs operators of $\fA^*$ to $A$ is denoted by $\WC(\fA^*,A)$.

It is easy to verify that $\F$ is \wcbs if and only if $\F$ is weak*\,continuous on the closed unit ball
$\fA^*_1$. For $A=\C$, we have the following interesting result; for a proof see \cite[Theorem 3.10.1]{Horvath}.

\begin{lem}[Banach-Dieudonn\'e]
\label{lem:Banach-Dieudonne}
 If a linear functional $\tau:\fA^*\to \C$ is \wcbs, then it is weak*\,continuous on the entire space
 $\fA^*$ so that $\tau=\hat u$, for some $u\in \fA$.
\end{lem}

Indeed, the lemma states that $\WC(\fA^*,\C)=\fA$ for every Banach space $\fA$.
Another consequence of the lemma is the following: if $\F:\fA^*\to \A$ is \wcbs, then
the composition $\phi\circ \F$, for every $\phi\in \A^*$, is weak*\,continuous on $\fA^*$,
so that $\phi\circ \F = \hat u$, for some $u\in \fA$. The converse, however, does not hold. That is,
$\F$ may fail to be \wcbs even if $\phi\circ \F$ is weak*\,continuous, for every $\phi\in A^*$;
see Example \ref{exa:T-is-not-wcbs-yet}.

We will use the following lemma several times; see \cite[Lemma 2.5]{Abtahi-RM}.

\begin{lem}\label{lem:w*-then-uniform}
 Let $(\psi_\la)$ be a bounded net in $\E^*$. If $\psi_\la\to \psi$ in the weak*\,topology,
 then $\psi_\la\to \psi$ uniformly on totally bounded sets in $\E$.
\end{lem}

The following characterization of \wcbs operators is very useful in our discussion.
Although proofs can be found in the literature, we provide one for the reader's convenience.

\begin{prop}\label{prop:dual-mapping}
  An operator $\F:\fA^*\to \A$ is
  \wcbs if and only if there is an operator $\G:\A^*\to \fA$ which is
  \wcbs, and
  \begin{equation}\label{eqn:F-G-dual-pair}
    \psi(\G \phi) = \phi(\F \psi) \quad (\psi\in \fA^*, \phi\in \A^*).
  \end{equation}
\end{prop}

\begin{proof}
  Suppose that $\F:\E^*\to \A$ is \wcbs. By Lemma \ref{lem:Banach-Dieudonne},
  for every $\phi\in \A^*$, the linear functional $\phi\circ \F$ is weak*\,continuous on $\E^*$
  so that $\phi\circ \F=u$, for some $u\in \E$. Define $\G:\A^*\to \E$ by $\G(\phi)=\phi\circ \F$.
  In fact, $\G=\F^*$, where $\F^*$ is the adjoint of $\F$ (e.g.\ see \cite[Theorem 4.10]{Rudin-FA}).
  We have
  \begin{equation}\label{eqn:F-G-dual-pair-in-proof}
     \psi(\G(\phi)) = \psi(\phi\circ \F) = \phi(\F(\psi)) \quad(\psi \in \E^*).
  \end{equation}

  We show that $\G$ is weak*\,continuous on $\A_1^*$. Let $(\phi_\la)$ be a net in $\A^*_1$
  that converges to $0$, in the weak*\,topology. Since $\E_1^*$ is weak*\,compact and $\F$ is
  weak*\,continuous on $\E^*_1$, the image set $\F(\E_1^*)$ is compact in $\A$.
  By Lemma \ref{lem:w*-then-uniform}, $\phi_\la \to 0$ uniformly on $\F(\E^*_1)$.
  Hence, for every $\e>0$, there is $\la_0$ such that
  \begin{equation*}
    |\psi(\G(\phi_\la))| = |\phi_\la(\F(\psi))|  \leq \e \quad (\psi\in \E_1^*, \la \geq \la_0).
  \end{equation*}
  Taking supremum over $\psi\in \E^*_1$, we get $\|\G(\phi_\la)\|\leq \e$, for $\la\geq \la_0$.
  Hence, $\G(\phi_\la)\to 0$ in $\E$ and thus $\G$ is \wcbs. The converse is proved similarly.
\end{proof}

We say that $\F$ and $\G$ form a \emph{dual pair} if they satisfy \eqref{eqn:F-G-dual-pair}.
In this case, we call $\G$ the \emph{dual mapping} of $\F$
(and $\F$ the dual mapping of $\G$).

\begin{rem}\label{rem:F-G-dual-pair}
  As it is mentioned in the proof, we get $\G=\F^*$ and thus
  \begin{enumerate}[(1)]
    \item $\|\F\|=\|\G\|$,
    \item $\F$ is compact if and only $\G$ is compact (\cite[Theorem 4.19]{Rudin-FA}).
  \end{enumerate}
\end{rem}

\begin{thm}\label{thm:wcbs-closed-subset-compact}
  $\WC(\fA^*,A)$ is a closed subspace of $\K(\fA^*,\A)$.
\end{thm}

\begin{proof}
   The inclusion $\WC(\fA^*,A)\subset \B(\fA^*,A)$ is obvious. By the Banach-Alaoglu theorem,
   $\fA^*_1$ is weak*\,compact. If $\F\in\WC(\fA^*,A)$ then $\F(\fA^*_1)$ is a compact set,
   meaning that $\F$ is a compact operator. Hence, $\WC(\fA^*,A)\subset \K(\fA^*,A)$.

   Assume that $(\F_n)$ is a sequence in $\WC(\fA^*,A)$ and that $\F_n\to \F$, with respect to the
   operator norm. This implies that $\F_n\to \F$ uniformly on $\fA^*_1$, and thus
   $\F$ is weak*\,continuous on $\fA^*_1$. Therefore, $\F\in \WC(\fA^*,A)$.
\end{proof}

The following example shows that compact operators may fail to be \wcbs.

\begin{ex}
  Let $\fA=c_0$ so that $\fA^*=\ell^1$. Let $\tau:\ell^1\to\C$ be the bounded linear
  functional induced by the element $\be=(1,1,\dotsc)$ of $\ell^\infty$;
  that is, $\tau(\al_1,\al_2,\dotsc) = \al_1+\al_2+\dotsb$. Then $\tau$ is compact.
  However, $\tau$ is not \wcbs, for $\be\notin c_0$; see Lemma \ref{lem:Banach-Dieudonne}.
\end{ex}

In the presence of dual mappings, compact operators are \wcbs. To see this, let
$\G:A^*\to \fA$ be the dual mapping of a compact operator $\F:\fA^*\to A$. Since $\F$ is compact,
the image set $\F(\fA^*_1)$ is totally bounded in $A$. The argument following \eqref{eqn:F-G-dual-pair-in-proof}
in the proof of Proposition \ref{prop:dual-mapping}, applies to $\G$, showing that
$\G$ is \wcbs. Now, by Proposition \ref{prop:dual-mapping}, $\F$ is \wcbs.
Below, we formally state this observation.

\begin{thm}\label{thm:wcbs-iff-compact}
  Let an operator $\F:\fA^*\to A$ admit a dual mapping satisfying \eqref{eqn:F-G-dual-pair}.
  If $\F$ is compact then $\F$ is \wcbs.
\end{thm}

The condition that an operator $\F:\fA^*\to A$ admits a dual mapping may fail to hold.
Indeed, we have the following characterization of this property.

\begin{prop}\label{prop:reflexive-iff}
  For a Banach space $\fA$, the following are equivalent;
  \begin{enumerate}[$(1)$]
    \item \label{item:fA-is-reflexive}
          $\fA$ is reflexive,
    \item \label{item:for-any-BS}
          every $\F\in\B(\fA^*,A)$ admits a dual mapping, for any Banach space $A$,
    \item \label{item:for-finite-dim-BS}
          every $\F\in\B(\fA^*,A)$ admits a dual mapping, for any finite dimensional Banach space $A$.
  \end{enumerate}
\end{prop}

\begin{proof}
  \eqref{item:fA-is-reflexive} \FI\ \eqref{item:for-any-BS}.
  Let $\F \in \B(\fA^*,A)$. For every $\phi\in A^*$, the composition
  $\phi\circ \F$ is a bounded linear functional. Since $\fA$ is reflexive, there is
  $u\in \fA$ such that $\phi\circ \F = \hat u$. Define $\G:A^*\to \fA$ by
  $\G(\phi)=\phi\circ \F$. Then $\psi(\G\phi)=\phi(\F\psi)$ for all
  $\phi\in A^*$ and $\psi\in \fA^*$. This means that $\G$ is the dual mapping of $\F$.

  \eqref{item:for-any-BS} \FI\ \eqref{item:for-finite-dim-BS}. Clear.

  \eqref{item:for-finite-dim-BS} \FI\ \eqref{item:fA-is-reflexive}.
  We show that if $\tau:\fA^*\to\C$ is a bounded linear functional
  then $\tau=\hat u$, for some $u\in \fA$. By the assumption, $\tau$ admits
  a dual mapping $\eta:\C\to\fA$. Let $u=\eta(1)$, so that $\eta(\al)=\al u$, for all $\al\in \C$.
  By \eqref{eqn:F-G-dual-pair}, we get $\psi(u) = \psi(\eta(1)) = \tau(\psi)$, for all $\psi\in \fA^*$. Hence $\tau=\hat u$.
\end{proof}

%On reflexive spaces, \wcbs operators coincide with compact operators.

\begin{cor}\label{cor:fA-is-reflexive-iff-W=B}
  For a Banach space $\fA$, the following are equivalent;
  \begin{enumerate}[$(1)$]
    \item $\fA$ is reflexive,
    \item $\WC(\fA^*,A)=\K(\fA^*,A)$, for any Banach space $A$,
    \item $\WC(\fA^*,A)=\B(\fA^*,A)$, for any finite-dimensional Banach space $A$.
  \end{enumerate}
\end{cor}

\begin{proof}
  It follows from Theorem \ref{thm:wcbs-iff-compact} and Proposition \ref{prop:reflexive-iff}.
\end{proof}

If $\fA=\ell^p$, $1<p<\infty$, then $\fA^*=\ell^q$, where $1/p+1/q=1$, and $\fA$ is reflexive.
In this case, $\WC(\ell^q,A)=\K(\ell^q,A)$, for all Banach space $A$. If $\fA=c_0$ then $\fA^*=\ell^1$ and
$\fA$ is not reflexive. By corollary above, $\WC(\ell^1,A)$ is a proper subspace of $\K(\ell^1,A)$,
for some Banach space $A$. In fact, this is the case for every Banach space $A$. To justify this,
let $\tau:\ell^1\to\C$ be a bounded linear functional that is not weak*\,continuous. Given a Banach space $A$,
take a nonzero element $a\in A$ and define $\F:\ell^1\to A$ by $\F(\be) = \tau(\be)a$. Clearly,
$\F$ is a compact operator that is not \wcbs.

The following example shows that a bounded operator $\F:\fA^*\to A$ may fail to be
\wcbs, even if $\phi\circ \F$, for every $\phi\in \A^*$, is weak*\,continuous on $\fA^*$.

\begin{ex}\label{exa:T-is-not-wcbs-yet}
  Let $\H$ be a Hilbert space. Every bounded operator $T:\H \to \H$  admits a dual mapping.
  Indeed, the equation $\ip{Tx,y} = \ip{x,T^*y}$, for $x,y\in \H$, shows that
  $T$ and $T^*$ form a dual pair. Therefore, $T$ is \wcbs if and only if $T$ is compact.
  Now, let $\H$ be infinite dimensional and let $T\in\B(\H)$ be a noncompact operator.
  By Corollary \ref{cor:fA-is-reflexive-iff-W=B}, $T$ is not \wcbs.
  However, for every $y\in \H$, the composition map $x\mapsto\ip{Tx,y}$ is weak*\,continuous.
\end{ex}

\section{Compact operators on $\ell^p$-spaces}% into Banach spaces}
\label{sec:compact}
Let $A$ be a Banach space, $p \in [1,\infty]$, and $1/p+1/q=1$.
In this section, we present a characterization of compact operators $\F:\ell^q \to A$.
To begin, let $A^\N$ denote the space of $A$-valued sequences $\a=(a_1,a_2,\dotsc)$.
For $p\in[1,\infty)$, define
\begin{equation}\label{eqn:p-norm-of-a}
  \|\a\|_p = \Bigprn{\sum_{i=1}^\infty \|a_i\|^p}^{1/p}.
\end{equation}
And, for $p=\infty$, define
\begin{equation}\label{eqn:p-norm-of-a-p=infty}
  \|\a\|_\infty = \sup\set{\|a_i\|: i\in \N}.
\end{equation}
For every $\phi\in A^*$, let $\phi(\a)=(\phi(a_1),\phi(a_2),\dotsc)$, and consider the following
spaces of $A$-valued sequences;
  \begin{align*}
    c_{00}(A) & = \set{\a\in A^\N : \text{$a_i=0$ eventually}},\\
    c_0(A) & = \set{\a\in A^\N : \text{$a_i \to 0$ in $A$}},\\
    \ell^p(A) & = \set{\a\in A^\N : \|\a\|_p < \infty}, \\
    \ell^p_c(A) & = \set{\a\in A^\N : \text{the set $\set{\phi(\a):\phi\in A^*_1}$ is totally bounded in $\ell^p$}},\\
    \ell^p_b(A) & = \set{\a\in A^\N : \text{the set $\set{\phi(\a):\phi\in A^*_1}$ is bounded in $\ell^p$}}.
  \end{align*}

\noindent
It is easy to see that
\begin{enumerate}[(1)]
  \item if $A$ is finite dimensional, then $\ell^p(A)= \ell^p_c(A) = \ell^p_b(A)$,
  \item $\ell^\infty(A)=\ell^\infty_b(A)$, for any Banach space $A$ (\cite[Theorem 3.18]{Rudin-FA}).
\end{enumerate}

\begin{prop}\label{prop:inclusions}
  If $p\in[1,\infty)$ then $\ell^p(A) \subset \ell^p_c(A)$.
\end{prop}

\begin{proof}
%  The inclusions $c_{00}(A) \subset \ell^p(A)$ and $\ell^p_c(A) \subset \ell^p_b(A)$
%  are obvious. We prove that $\ell^p(A) \subset \ell^p_c(A)$.
  Let $\a=(a_1,a_2,\dotsc)$ be an element of $\ell^p(A)$, and set
  $K=\set{\phi(\a):\phi\in A^*_1}$. We show that $K$ is totally bounded in $\ell^p$.
  Since $\|\a\|_p<\infty$, given $\e>0$, there is $m\in \N$ such that
  $\sum_{i > m} \|a_i\|^p < \e^p$. Hence,
  \begin{equation*}
    \Bigprn{\sum_{i > m} |\phi(a_i)|^p}^{1/p} < \e \quad (\phi\in A^*_1).
  \end{equation*}
  Hence, $K$ satisfies \eqref{eqn:Kolmogorov-Riesz-compactness}. Obviously, $K$ is bounded
  in $\ell^p$. By Theorem \ref{thm:Kolmogorov-Riesz-compactness},
  $K$ is totally bounded in $\ell^p$, meaning that $\a \in \ell^p_c(A)$.
\end{proof}

The following two examples show that, in general, the inclusions
$\ell^p(A) \subset \ell^p_c(A)\subset \ell^p_b(A)$
are proper inclusions.

\begin{ex}
  Let $A = \ell^2$ and $p=2$. For every $i\in \N$, let $a_i = e_i/\sqrt{i}$, where $\{e_1,e_2,\dotsc\}$
  is the standard unit vector basis for $\ell^2$, and set $\a=(a_1,a_2,\dotsc)$. Then $\|a_i\| = 1/\sqrt{i}$
  so that $\sum_{i=1}^\infty \|a_i\|^2$ does not converge, meaning that $\a\notin \ell^2(A)$.
  We show that $\a\in \ell^2_c(A)$. Let $\be=(\be_1,\be_2,\dotsc)$ be an element of $(\ell^2)^*=\ell^2$,
  with $\|\be\|_2 \leq 1$. Then $\hat\be(a_i) = \be_i/\sqrt{i}$, for every $i$, and thus
  \begin{equation*}
    \sum_{i>m} |\hat\be(a_i)|^2
     = \sum_{i>m} \frac{|\be_i|^2}{i} < \frac1m \sum_{i>m} |\be_i|^2 \leq \frac1m\|\be\|^2_2 \leq \frac1m.
  \end{equation*}

  Satisfying all the conditions in Theorem \ref{eqn:Kolmogorov-Riesz-compactness},
  the set $\set{\hat\be(\a):\|\be\|_2\leq 1}$ is totally bounded in $\ell^2$.
  Therefore, $\a\in\ell^2_c(A)$.
\end{ex}

\begin{ex}
  Let $A = \ell^2$ and let $\a=(e_1,e_2,\dotsc)$. Then $\a\in \ell^p_b(A)$, $p\in[1,\infty)$.
  However, %$\hat\be(a_i)=\be_i$, for every $\be=(\be_1,\be_2,\dotsc)$ in $A^*=\ell^2$. In particular, for $\be=e_n$, we get
  \begin{equation*}
    \sum_{i > m} |\hat e_n(e_i)|^p = 1 \quad (n > m).
  \end{equation*}
  This shows that $\a\notin \ell^p_c(A)$.
\end{ex}

\begin{defn}\label{dfn:normm(a)}
  Let $A$ be a Banach space and $1 \leq p\leq \infty$.
  For every $\a\in \ell^p_b(A)$, define
  \begin{equation}\label{eqn:normm{a}}
    \normm{\a}_p = \sup\set{\|\phi(\a)\|_p : \phi\in A^*, \|\phi\|\leq 1}.
  \end{equation}
\end{defn}

We will see that $\enormm_p$ is a complete norm on $\ell^p_b(A)$, and that
$\ell^p_c(A)$ is a closed subspace of $\ell^p_b(A)$. If $p\in[1,\infty)$ then
$\ell^p_c(A)=\ol{c_{00}(A)}$, where the closure is taken with respect to $\enormm_p$.
If $p=\infty$ then $\normm{\a}_\infty = \|\a\|_\infty$, and
$\ol{c_{00}(A)}=c_0(A) \subset \ell^\infty_c(A)$.
Before discussing any of these, let us present a characterization of bounded
operators on $\ell^q$, $1\leq q <\infty$.
Throughout, given an element $\a=(a_1,a_2,\dotsc)$ of $\ell^p_b(A)$,
define
   \begin{equation}\label{eqn:a-induces-Fa}
      \F_\a(\beta) = \sum_{i=1}^\infty \beta_i a_i,
   \end{equation}
for every $\beta = (\beta_1,\beta_2,\dotsc)$ in $\ell^q$.
If $\a\in\ell^1_b(A)$, the equation is well-defined for $\be\in c_0$.

\begin{thm}\label{thm:bounded-iff}
  Let $1<p \leq \infty$ and $1/p+1/q=1$.
  \begin{enumerate}[$(1)$]
    \item \label{item:every-a-indices-F:lq-to-A}
          Every $\a\in \ell^p_b(A)$ induces a bounded operator $\F_\a: \ell^q \to A$ defined by
          \eqref{eqn:a-induces-Fa} with $\|\F_\a\| = \normm{\a}_p$.
    \item \label{item:every-F-is-of-the-form-Fa}
        If $\F:\ell^q\to A$ is a bounded operator, then $\F=\F_\a$ for some $\a\in\ell^p_b(A)$.

%    \item \label{item:F-is-compact-iff-a-in-lpc}
%          $\F_\a$ is compact if and only if $\a\in \ell^p_c(A)$.
  \end{enumerate}
\end{thm}

\begin{proof}
  \eqref{item:every-a-indices-F:lq-to-A}.
  Let $\a=(a_1,a_2,\dotsc)$ be an element of $\ell^p_b(A)$. For $n\geq m \geq 1$, we get
  \begin{align}
    \Bignorm{\sum_{i=m}^n \beta_i a_i}
      & = \sup_{\phi\in A^*_1} \Bigabs{\phi\Bigprn{\sum_{i=m}^n \beta_i a_i}}
        = \sup_{\phi\in A^*_1} \Bigabs{\sum_{i=m}^n \beta_i \phi(a_i)} \label{eqn:norm(F(be))}\\
%      & \leq \sup_{\phi\in A^*_1} \Bigprn{\sum_{i=m}^n|\phi(a_i)|^p}^{1/p} \Bigprn{\sum_{i=m}^n |\beta_i|^q}^{1/q} \\
      & \leq\sup_{\phi\in A^*_1} \|\phi(\a)\|_p \Bigprn{\sum_{i=m}^n |\beta_i|^q}^{1/q}
        = \normm{\a}_p \Bigprn{\sum_{i=m}^n |\beta_i|^q}^{1/q}.
  \end{align}

  This shows that the series in \eqref{eqn:a-induces-Fa} satisfies the Cauchy criterion in the Banach space $A$, and thus it converges
  to $\F_\a(\be)$. It also implies that $\|\F_\a(\be)\| \leq \normm{\a}_p \|\be\|_q$,
  for all $\be \in \ell^q$ so that $\|\F_\a\|\leq \normm{\a}_p$. In fact, $\|\F_\a\|=\normm{\a}_p$, for
  \begin{align*}
    \|\F_\a\|
        = \sup_{\|\be\|_q \leq 1} \|\F_\a(\be)\|
       & = \sup_{\|\be\|_q \leq 1} \Bignorm{\sum_{i=1}^\infty \beta_i a_i}
       = \sup_{\|\be\|_q \leq 1} \sup_{\phi\in A^*_1} \Bigabs{\sum_{i=1}^\infty \beta_i \phi(a_i)} \\
      & = \sup_{\phi\in A^*_1} \sup_{\|\be\|_q \leq 1} \Bigabs{\sum_{i=1}^\infty \beta_i \phi(a_i)}
        = \sup_{\phi\in A^*_1} \|\phi(\a)\|_p  %\Bigprn{\sum_{i=1}^\infty |\phi(a_i)|}^{1/p}
        = \normm{\a}_p.
  \end{align*}

    \eqref{item:every-F-is-of-the-form-Fa}.
    Let $\F:\ell^q \to A$ be a bounded operator. For every $i\in\N$, let $a_i=\F(e_i)$, where
    $\{e_1,e_2,\dotsc\}$ is the standard unit vector basis, and let $\a=(a_1,a_2,\dotsc)$. We show
    that $\a\in\ell^p_b(A)$. First, let $1<p<\infty$. Given $\phi\in A^*_1$,
    for every $n\in\N$, we have
    \begin{equation}\label{eqn:a-in-lpb}
    \begin{split}
      \Bigprn{\sum_{i=1}^n |\phi(a_i)|^p}^{1/p}
        & = \sup\Bigset{\Bigabs{\sum_{i=1}^n\be_i\phi(a_i)}: \be\in \ell^q,\|\be\|_q \leq 1} \\
        & = \sup\Bigset{\Bigabs{\phi\Bigprn{\sum_{i=1}^n\be_i\F(e_i)}}: \be\in \ell^q,\|\be\|_q \leq 1} \\
        & \leq \sup\Bigset{\Bignorm{\F\Bigprn{\sum_{i=1}^n\be_ie_i}}: \be\in \ell^q,\|\be\|_q \leq 1}  \\
        & \leq \|\F\| \sup\Bigset{\Bignorm{\sum_{i=1}^n\be_ie_i}_q: \be\in \ell^q,\|\be\|_q \leq 1} \\
        & \leq \|\F\|. %\sup\set{\|\be\|_q: \be\in \ell^q,\|\be\|_q \leq 1}
          % = \|\F\|.
    \end{split}
    \end{equation}

    Therefore, $\phi(\a)\in \ell^p$ and $\|\phi(\a)\|_p \leq \|\F\|$, for all $\phi\in A^*_1$,
    showing that $\a\in \ell^p_b(A)$. In case $p=\infty$, we get $\|a_i\| = \|\F(e_i)\| \leq \|\F\|$, for $i\in\N$,
    and thus $\a\in \ell^\infty(A)$.

    In both cases, since $\F(e_i) = a_i = \F_\a(e_i)$, for every $i\in \N$,
    we get $\F=\F_\a$.
\end{proof}

We are now in a position to present a characterization of compact operators.

\begin{thm}\label{thm:compact-iff}
  Let $1<p\leq \infty$ and $1/p+1/q=1$. An operator $\F:\ell^q \to A$ is compact if and only if $\F=\F_\a$,
  for some $\a\in\ell^p_c(A)$.
\end{thm}

\begin{proof}
  By Theorem \ref{thm:bounded-iff}, $\F=\F_\a$, for some $\a\in \ell^p_b(A)$.
  Define $\G_\a:A^*\to \ell^p$ by $\G_\a(\phi)=\phi(\a)$.
  Then, for every $\be\in \ell^q$ and $\phi\in A^*$,
  \begin{equation*}
    \hat\be(\G_\a(\phi))
     = \hat\be(\phi(\a))
     = \sum_{i=1}^\infty \be_i \phi(a_i)
     =  \phi\Bigprn{\sum_{i=1}^\infty \be_i a_i}
     = \phi(\F_\a(\be)).
  \end{equation*}

  Therefore, $\G_\a=\F_\a^*$, the adjoint of $\F_\a$. By \cite[Theorem 4.19]{Rudin-FA},
  $\F_\a$ is compact if and only if $\G_\a$ is compact, equivalently,
  $\G_\a(A^*_1)=\set{\phi(\a) : \phi\in A^*_1}$ is totally bounded in $\ell^p$.
  This means that $\a\in\ell^p_c(A)$.
\end{proof}

We now discuss the case of $p=1$.

\begin{thm}\label{thm:bounded-compact-p=1}
  Every $\a\in\ell^1_b(A)$ induces a bounded operator $\F_\a:c_0\to A$,
  defined by \eqref{eqn:a-induces-Fa}, with $\|\F_\a\|=\|\a\|_1$. Conversely, if
  $\F:c_0 \to A$ is a bounded operator then $\F=\F_\a$, for some $\a\in\ell^1_b(A)$.
  Moreover, $\F$ is compact if and only if $\a\in\ell^1_c(A)$.
\end{thm}

\begin{proof}
    Assume that
     $\a=(a_1,a_2,\dotsc)$ is an element of $\ell^1_b(A)$. For $n\geq m \geq 1$,
    similar calculations as in \eqref{eqn:norm(F(be))} lead to
  \begin{align*}
    \Bignorm{\sum_{i=m}^n \beta_i a_i}
        \leq  \normm{\a}_1 \max\set{|\be_i|:m\leq i \leq n}.
  \end{align*}

  Since $\be_m\to0$ as $m\to \infty$, the series in \eqref{eqn:a-induces-Fa} satisfies the Cauchy criterion,
  and thus it converges in $A$ to $\F_\a(\be)$. It also follows that $\|\F_\a(\be)\| \leq \normm{\a}_1 \|\be\|_\infty$,
  for all $\be\in c_0$ so that $\|\F_\a\|\leq \normm{\a}_1$. It is easy to verify that $\|\F_\a\|=\normm{\a}_\infty$.

  Conversely, let $\F:c_0 \to A$ be a bounded operator and let $a_i=\F(e_i)$, for $i\in\N$.
  Calculations in \eqref{eqn:a-in-lpb} remain valid for $p=1$, and lead to
    \begin{equation*}
      \sum_{i=1}^n |\phi(a_i)|\leq \|\F\|
      \quad (n\in \N,\phi\in A^*_1).
    \end{equation*}

    Therefore, $\phi(\a)\in \ell^1$ and $\|\phi(\a)\|_1 \leq \|\F\|$, for all $\phi\in A^*_1$,
    showing that $\a\in \ell^1_b(A)$. Since $\F(e_i) = a_i = \F_\a(e_i)$, for every $i\in \N$,
    we get $\F=\F_\a$.

    Finally, define $\G_\a:A^*\to\ell^1$ by $\G_\a(\phi)=\phi(\a)$. Then $\G_\a=\F_\a^*$ and,
    therefore, $\F_\a$ is compact if and only if $\G_\a$ is compact, equivalently,
    $\G_\a(A^*_1)$ is totally bounded in $\ell^1$. This means that $\a\in\ell^1_c(A)$.
\end{proof}

We summarize the discussion above in the following statement.

\begin{thm}\label{thm:summarized}
  Let $1 \leq p\leq \infty$. Then $(\ell^p_b(A),\enormm_p)$ is a Banach space having $\ell^p_c(A)$
  as a closed subspace. In addition, if $\F_\a$ is defined by \eqref{eqn:a-induces-Fa}
  then
  \begin{enumerate}[$(1)$]
    \item for $p=1$, the map $\a\mapsto \F_\a$ is an isometric isomorphism
          of $\ell^1_b(A)$ onto $\B(c_0,A)$, mapping $\ell^1_c(A)$ onto $\K(c_0,A)$,

    \item for $1<p \leq \infty$, the map $\a\mapsto \F_\a$ is an isometric isomorphism
          of $\ell^p_b(A)$ onto $\B(\ell^q,A)$, mapping $\ell^p_c(A)$ onto $\K(\ell^q,A)$,
          where $1/p+1/q=1$.
  \end{enumerate}
\end{thm}

The rest of the section is devoted to other characterizations of $\ell^p_c(A)$.

\begin{prop}\label{prop:a-in-lc-iff-an-to-a}
  Let $p\in[1,\infty)$. Given $\a\in \ell^p_b(A)$, define
  \begin{equation}\label{eqn:an}
    \a_n = (a_1,\dotsc,a_n,0,0,0,\dotsc) \quad (n\in\N).
  \end{equation}
  Then $\a\in \ell^p_c(A)$ if and only if
  $\a_n\to\a$, with respect to $\enormm_p$. In particular, $\ell^p_c(A)$ is the closure of $c_{00}(A)$ in $\ell^p_b(A)$.
\end{prop}

\begin{proof}
  Since $c_{00}(A)\subset \ell^p_c(A)$, we get $\ol{c_{00}(A)}\subset \ell^p_c(A)$,
  where the closure is taken with respect to $\enormm_p$.
  Hence, $\a_n\to \a$ implies that $\a\in \ol{c_{00}(A)}\subset \ell^p_c(A)$.

  Conversely, if $\a\in \ell^p_c(A)$, then $\set{\phi(\a):\phi\in A^*_1}$ is a totally bounded
  set in $\ell^p$. By Theorem \ref{thm:Kolmogorov-Riesz-compactness}, for every $\e>0$,
  there is $m\in \N$ such that
    \begin{equation*}
      \Bigprn{\sum_{i > m} |\phi(a_i)|^p}^{1/p} < \e \quad (\phi \in A^*_1).
    \end{equation*}
  On the other hand, for $n\geq m$, we have
  \begin{equation*}
    \normm{\a_n-\a}_p
      = \sup_{\phi\in A^*_1} \Bigprn{\sum_{i > n} |\phi(a_i)|^p}^{1/p}
      \leq \sup_{\phi\in A^*_1} \Bigprn{\sum_{i > m} |\phi(a_i)|^p}^{1/p}
      \leq \e.
  \end{equation*}
  This shows that $\normm{\a_n-\a}_p\to0$ as $n\to \infty$.
\end{proof}

If $p=\infty$, then $\ell^\infty_b(A)=\ell^\infty(A)$ and $\normm{\a}_\infty=\|\a\|_\infty$,
for every $\a\in\ell^\infty(A)$. In this case, $\ol{c_{00}(A)} = c_0(A)$ and the characterization
above is not valid for $p=\infty$. Instead, we have the following result for $\ell^\infty_c(A)$
the proof of which is postponed to the end of Section \ref{sec:C(Omega)}.

\begin{prop}\label{prop:a-in-lc-iff-p=infty}
  Let $\a=(a_1,a_2,\dotsc)$ be a sequence in $A$. Then $\a\in\ell^\infty_c(A)$ if and only if the
  set $\set{a_1,a_2,\dotsc}$ is totally bounded in $A$.
\end{prop}

The condition $\a_n\to \a$ in $\ell^p_b(A)$, $1\leq p<\infty$, implies that
the operator $\F_\a$ in \eqref{eqn:a-induces-Fa} is compact. The question that arises here,
for $p=\infty$, is that how the operator $\F_\a$ responds to the condition $\a_n\to \a$ in $\ell^\infty(A)$.
%The following result answers this question.

\begin{thm}\label{thm:F=Fa-case-q=1}
   Let $\a=(a_1,a_2,\dotsc)$ be an element of $\ell^\infty(A)$, and define $\a_n$ as in \eqref{eqn:an}.
   The following statements are equivalent;
   \begin{enumerate}[$(1)$]
     \item the induced operator $\F_\a:\ell^1\to A$ is \wcbs,
     \item $\a_n\to \a$ in $\ell^\infty(A)$, equivalently, $\a\in c_0(A)$.
   \end{enumerate}
\end{thm}

\begin{proof}
  Assume that $\F_\a$ is \wcbs. Consider the bounded sequence $(e_i)$ of linear functionals
  on $c_0$ defined by $e_i(\al) = \al_i$, for $\al=(\al_1,\al_2,\dotsc)$ in $c_0$.
  Then $e_i(\al) \to 0$ for every $\al\in c_0$, meaning that $e_i\to 0$, in the weak*
  topology of $\ell^1$. Since $\F_\a$ is \wcbs, we get $\F_\a(e_i) \to 0$ in $A$.
  Therefore, $a_i\to 0$ and thus $\a\in c_0(A)$.
  In particular, $\a_n\to \a$ in $\ell^\infty(A)$.

  Conversely, assume that $\a_n\to \a$ in $\ell^\infty(A)$. Therefore,
  $a_i\to 0$ in $A$ and thus $\phi(\a)\in c_0$, for all $\phi\in A^*$.
  Define $\G_\a:A^*\to c_0$ by $\G_\a(\phi)=\phi(\a)$. The operators $\F_\a$ and $\G_\a$
  satisfy \eqref{eqn:F-G-dual-pair} so that they form a dual pair.
  On the other hand, since $a_i\to0$, the set $\set{a_1,a_2,\dotsc}$ is totally bounded in $A$.
  By Proposition \ref{prop:a-in-lc-iff-p=infty}, $\a\in\ell^\infty_c(A)$ and thus $\F_\a$ is compact.
  Now, by Theorem \ref{thm:wcbs-iff-compact},
  $\F_\a$ is \wcbs.
\end{proof}

For a nonzero element $a\in A$, let $\a=(a,a,\dotsc)$. Then $\a\in \ell^\infty_c(A)\setminus c_0(A)$.
If $A$ is infinite dimensional, take a sequence $\a=(a_1,a_2,\dotsc)$ of linearly independent unit vectors
of $A$. Then $\a \in \ell^\infty(A)$. Since the set $\set{a_1,a_2,\dotsc}$ is not totally bounded in $A$
we get $\a \notin \ell^\infty_c(A)$. These examples show that, in general, the inclusions
$c_0(A) \subset \ell^\infty_c(A) \subset \ell^\infty(A)$ are proper.

\section{$\c(\Om)$-valued operators}
\label{sec:C(Omega)}
\renewcommand{\A}{\scrA}

In this section, we specifically focus on the Banach space of continuous functions.
Let $\Om$ be a compact Hausdorff space, and let $\c(\Om)$ denote the space of
continuous functions $f:\Om\to\C$ equipped with the uniform norm;
\begin{equation*}
  \|f\|_\Om = \sup\set{|f(s)|: s\in \Om}.
\end{equation*}

The dual space $\c(\Om)^*$ is identified as the space $\M(\Om)$ of complex Radon measures
on $\Om$. To simplify the notation, let $\A=\c(\Om)$ and $\M=\M(\Om)$.
For $f\in \A$ and $\mu \in \M$, let
\begin{equation}\label{eqn:mu(f)}
  \mu(f) = \int_\Om f\diff \mu.
\end{equation}
In particular, if $\de_s$ is the point mass measure at $s\in\Om$, then $\de_s(f)=f(s)$.

Let $F=(f_1,f_2,\dotsc)$ be an element of $\ell^p_b(\A)$, $1\leq p\leq \infty$.
For every $s\in \Om$, let
\begin{equation}\label{eqn:F(s)}
   F(s) = (f_1(s),f_2(s),\dotsc).
\end{equation}

This allows us to consider $F$ as a function of $\Om$ to $\ell^p$.
The function $F$ is bounded, but it might fail to be continuous.
We will prove that $F$ is continuous if and only if $F(\Om)$ is totally bounded
in $\ell^p$. First, we have the following lemma.

\begin{lem}\label{lem:F-in-lpc-iff-F(X)-is-compact}
  Let $F=(f_1,f_2,\dotsc)$ belong to $\ell^p_b(\A)$, $1\leq p\leq \infty$.
  Then $F\in\ell^p_c(\A)$ if and only if $F(\Om)$ is totally bounded in $\ell^p$.
\end{lem}

\begin{proof}
  For every $\mu\in \M$, let $\mu(F)=(\mu(f_1),\mu(f_2),\dotsc)$ and set
  \begin{equation*}
    K=\set{\mu(F): \mu\in\M,\|\mu\| \leq 1}.
  \end{equation*}\

  Then $F(\Om)\subset K$. Now, if $F\in\ell^p_c(\A)$ then $K$ is totally bounded in $\ell^p$ and
  so is $F(\Om)$. Conversely, assume that $F(\Om)$ is totally bounded in $\ell^p$.
  By \cite[Theorem 3.20]{Rudin-FA}, the closed convex hull $\cco F(\Om)$ is compact in $\ell^p$.
  Let $\mu\in \M$ be a real Radon measure on $\Om$ with $\|\mu\|\leq 1$. Using Jordan decomposition
  theorem, write $\mu=\mu^+ - \mu^-$ and $\mu(F)=\mu^+(F)-\mu^-(F)$.
  By \cite[Theorem 3.27]{Rudin-FA}, $\mu^+(F)\in \|\mu^+\| \cco F(\Om)$ and
  $\mu^-(F)\in \|\mu^-\| \cco F(\Om)$. In general, if $\mu=\mu_1+\I \mu_2$ is a complex
  Radon measure on $\Om$ with $\|\mu\|\leq 1$, then $\mu(F)=\mu_1(F) + \I \mu_2(F)$,
  and $\mu(F)\in K'$ where
  \begin{equation*}
    K' = \bigset{s_1 u_1 - s_2 u_2 + \I (t_1 v_1 - t_2v_2): s_i,t_i,\in [0,1], u_i,v_i\in \cco F(\Om),i=1,2}.
  \end{equation*}
  Since $K'$ is compact and $K\subset K'$, the set $K$ is totally bounded in $\ell^p$ and
  thus $F\in \ell^p_c(\A)$.
\end{proof}

Given $F\in \ell^p_b(\A)$, define
$\F_F:\ell^q\to \A$ by
\begin{equation}\label{eqn:FF}
  \F_F(\be) = \sum_{i=1}^\infty \be_i f_i.
\end{equation}

If $p=1$, the equation above is well-defined for $\be\in c_0$.

\begin{thm}\label{thm:FF-is-compact}
  Let $p\in[1,\infty)$ and $F\in \ell^p_b(A)$. The following are equivalent;
  \begin{enumerate}[$(1)$]
    \item \label{item:FF-is-compact}
          the operator $\F_F$, defined by \eqref{eqn:FF}, is compact,

    \item \label{item:F-in-lpc}
          the sequence $F=(f_1,f_2,\dotsc)$ belongs to $\ell^p_c(\A)$,

    \item \label{item:F(X)-is-compact-in-lp}
          the image set $F(\Om)$ is totally bounded in $\ell^p$,

    \item \label{item:F:X-to-lp-is-cnts}
          the function $F$, defined by \eqref{eqn:F(s)}, is continuous.
  \end{enumerate}
\end{thm}

\begin{proof}
  \eqref{item:FF-is-compact} \IFF\ \eqref{item:F-in-lpc}.
  It follows from Theorems \ref{thm:compact-iff} and \ref{thm:bounded-compact-p=1}.

  \eqref{item:F-in-lpc} \IFF\ \eqref{item:F(X)-is-compact-in-lp}.
  It follows from Lemma \ref{lem:F-in-lpc-iff-F(X)-is-compact}.

  \eqref{item:F(X)-is-compact-in-lp} \FI\ \eqref{item:F:X-to-lp-is-cnts}.
  Assume that $F(\Om)$ is totally bounded in $\ell^p$. We show that $F$ is continuous at each point
  $s_0\in \Om$. Given $\e>0$, by Theorem \ref{thm:Kolmogorov-Riesz-compactness},
  there is $m>0$ such that
  \begin{equation*}
    \sum_{i>m} |f_i(s)|^p < \e^p \quad (s\in \Om).
  \end{equation*}

  There is a neighbourhood $U_0$ of
  $s_0$ in $\Om$ such that
  \begin{equation*}
    |f_i(s)-f_i(s_0)| \leq \frac{\e}{m^{1/p}} \quad (s\in U_0, 1\leq i \leq m).
  \end{equation*}

  Hence, for every $s\in U_0$, we get
  \begin{align*}
    \|F(s)-F(s_0)\|^p_p
       & = \sum_{i=1}^m |f_i(s) - f_i(s_0)|^p + \sum_{i>m} |f_i(s) - f_i(s_0)|^p\\
       & \leq \sum_{i=1}^m \frac{\e^p}m + 2^p \sum_{i>m} \bigprn{|f_i(s)|^p + |f_i(s_0)|^p} \\
       & \leq \e^p + 2^p(\e^p+\e^p) = (1 + 2^{p+1})\e^p.
  \end{align*}

  It shows that $F$ is continuous at $s_0$.

  \eqref{item:F:X-to-lp-is-cnts} \FI\ \eqref{item:F(X)-is-compact-in-lp}.
  Obvious.
\end{proof}

We now discuss the case where $p=\infty$.

\begin{thm}\label{thm:FF-is-compact-p=infty}
  For $F \in \ell^\infty(\A)$, the following are equivalent.
  \begin{enumerate}[$(1)$]
    \item \label{item:F(X)-is-totbdd-in-ell-infty}
          the image set $F(\Om)$ is totally bounded in $\ell^\infty$,

    \item \label{item:F-in-lc-p=infty}
          the sequence $F=(f_1,f_2,\dotsc)$ belongs to $\ell^\infty_c(\A)$,

    \item \label{item:FF-is-compact-p=infty}
          the operator $\F_F:\ell^1\to \A$ is compact,

    \item \label{item:F-is-totbdd-p=infty}
          the set $\cF=\set{f_1,f_2,\dotsc}$ is totally bounded in $\A$,

    \item \label{item:F-is-cnts-p=infty}
          the function $F:\Om \to\ell^\infty$ is continuous.
  \end{enumerate}
\end{thm}

\begin{proof}
  \eqref{item:F(X)-is-totbdd-in-ell-infty} \IFF\ \eqref{item:F-in-lc-p=infty}.
  It follows from Lemma \ref{lem:F-in-lpc-iff-F(X)-is-compact}.

  \eqref{item:F-in-lc-p=infty} \IFF\ \eqref{item:FF-is-compact-p=infty}.
  It follows from Theorem \ref{thm:compact-iff}.

  \eqref{item:FF-is-compact-p=infty} \FI\ \eqref{item:F-is-totbdd-p=infty}.
  Let $\cG = \set{\F_F(\be): \|\be\|_1 \leq 1}$. For every $i$,
  we have $f_i = \F_F(e_i)$, where $\set{e_1,e_2,\dotsc}$ is the standard unit vector basis for $\ell^p$.
  Therefore, $\cF\subset \cG$. Now, if $\F_F$ is a compact operator,
  then $\cG$ is a totally bounded set in $\A$. This implies that $\cF$ is also totally bounded in $\A$.

  \eqref{item:F-is-totbdd-p=infty} \FI\ \eqref{item:F-is-cnts-p=infty}.
  Let $\cF$ be a totally bounded set in $\A$.
  Given $s_0\in \Om$, by the Arzel\'a-Ascoli theorem, $\cF$ is equicontinuous at $s_0$;
  that is, for every $\e>0$, there is a neighbourhood $U$ of $s_0$ in $\Om$ such that
  \begin{equation*}
     |f_i(s) - f_i(s_0)| \leq \e \quad (s\in U,i\in\N).
  \end{equation*}

  Taking supremum over all $i\in\N$, we get $\|F(s)-F(s_0)\|_\infty\leq \e$ for all $s\in U$.
  This shows that $F$ is continuous at $s_0$.

  \eqref{item:F-is-cnts-p=infty} \FI\ \eqref{item:F(X)-is-totbdd-in-ell-infty}.
  Obvious.
\end{proof}

Although the following statement is part of the preceding discussion, we believe it
deserves to be presented as an independent result.

\begin{cor}\label{cor:F-is-cnts-iff-F(X)-is-ttbdd}
  Let $\Om$ be a compact Hausdorff space, and let $F:\Om\to\ell^p$, $1\leq p \leq \infty$,
  be a function with continuous components $f_1,f_2,\dotsc$. Then $F$ is continuous if and only if
  $F(\Om)$ is totally bounded in $\ell^p$.
\end{cor}

Finally, based on Theorem \ref{thm:FF-is-compact-p=infty}, we present a proof of Proposition \ref{prop:a-in-lc-iff-p=infty}.

\begin{proof}[Proof of Proposition $\ref{prop:a-in-lc-iff-p=infty}$]
  By the Banach-Alaoglu theorem, the closed unit ball $A^*_1$ is a compact Hausdorff space
  for the weak*\,topology. Let $\A=\c(A^*_1)$. For every $a\in A$, define $\hat a:A^*_1\to \C$, by
  $\hat a(\phi)=\phi(a)$. Then $\hat a\in \A$, $\|\hat a\|=\|a\|$ and the map $A \to \A$, $a\mapsto \hat a$,
  embeds $A$ as a closed subspace of $\A$. Let $\cF$ be a subset of $A$ and $\F:\ell^1\to A$ be a linear
  operator. The following statements hold;
  \begin{enumerate}[(1)]
    \item $\cF$ is totally bounded in $A$ if and only if $\cF$ is totally bounded in $\A$,
    \item $\F:\ell^1\to A$ is compact if and only if $\F:\ell^1\to \A$ is compact.
  \end{enumerate}

  For $\a=(a_1,a_2,\dotsc)$, let $\hat \a=(\hat a_1,\hat a_2,\dotsc)$.
  The observations above, combined with Theorem \ref{thm:FF-is-compact-p=infty}, demonstrate that
  the following statements are equivalent;
  \begin{enumerate}[(1)]
    \item $\a\in \ell^\infty_c(A)$,
    \item $\hat\a\in \ell^\infty_c(\A)$,
    \item the set $\set{\hat a_1,\hat a_2,\dotsc}$ is totally bounded in $\A$,
    \item the set $\set{a_1,a_2,\dotsc}$ is totally bounded in $A$.
  \end{enumerate}
  This completes the proof.
\end{proof}

\section{$\B(\H)$-valued operators}
\label{sec:B(H)}
\renewcommand{\A}{\cA}
In this section, we investigate the specific case of $\cA=\B(\H)$, where $\H$ is a Hilbert space.
The discussion is limited to $1<p<\infty$ so that $\fA=\ell^p$ is reflexive.
Given an element $\T=(T_1,T_2,\dotsc)$ of $\ell^p_b(\A)$, we use the same notation $\T$
to denote the induced operator;
\begin{equation}\label{eqn:TT}
  \T:\ell^q \to \A, \quad
  \T\be = \sum_{i=1}^\infty \be_i T_i.
\end{equation}

For $x,y\in \H$, define $\ip{\T x,y} = \bigprn{\ip{T_1x,y},\ip{T_2x,y},\dotsc}$.
If $\be=(\be_1,\be_2,\dotsc)$ is an element of $\ell^q$, then
\begin{equation*}
       \hat \be \ip{\T x,y} = \sum_{i=1}^\infty \be_i \ip{T_ix,y} = \ip{(\T\be)x,y}.
\end{equation*}
Moreover,
\begin{equation}\label{eqn:norm(Tbe)}
  \|\T\be\| = \sup\set{\|\ip{(\T\be) x,y}\|:\|x\|=\|y\|=1}.
\end{equation}
It is worth noting that the operator norm of $\T$ satisfies the following equality;
\begin{equation*}
  \|\T\| = \sup\set{\|\ip{\T x,y}\|:\|x\|=\|y\|=1}.
\end{equation*}

\begin{thm}\label{thm:TT-is-compact}
  For $\T\in\ell^p_b(\A)$, $1 < p < \infty$, the following are equivalent;
  \begin{enumerate}[$(1)$]
    \item \label{item:TT-is-compact}
          the operator $\T$, defined by \eqref{eqn:TT}, is compact,

    \item \label{item:T-in-lpc}
          the sequence $\T=(T_1,T_2,\dotsc)$ belongs to $\ell^p_c(\A)$,

    \item \label{item:<Tx,y>-is-totbdd}
          the set $\cX=\set{\ip{\T x, y}:\|x\|=\|y\|=1}$ is totally bounded in $\ell^p$,

    \item \label{item:<Tx,x>-is-totbdd}
          the set $\W=\set{\ip{\T x, x}:\|x\|=1}$ is totally bounded in $\ell^p$.
  \end{enumerate}
\end{thm}

\begin{proof}
  \eqref{item:TT-is-compact} \IFF\ \eqref{item:T-in-lpc}.
  It follows from Theorem \ref{thm:compact-iff}.

  \eqref{item:T-in-lpc} \FI\ \eqref{item:<Tx,y>-is-totbdd}.
  For $\phi\in \A^*$, let $\phi(\T)=(\phi(T_1),\phi(T_2),\dotsc)$ and set
  \begin{equation*}
     K=\set{\phi(\T): \phi\in \A^*_1}.
  \end{equation*}

  If $\T\in\ell^p_c(\A)$ then $K$ is totally bounded in $\ell^p$. For unit vectors $x,y\in \H$, define
  $\phi_{x,y}(T)=\ip{Tx,y}$, for all $T\in\A$. Then $\phi_{x,y}\in \A^*_1$ and $\phi_{x,y}(\T) = \ip{\T x,y}$.
  Hence, $\cX \subset K$ and thus $\cX$ is totally bounded.

  \eqref{item:<Tx,y>-is-totbdd} \FI\ \eqref{item:<Tx,x>-is-totbdd}.
  This is obvious, for $\W\subset \cX$.

  \eqref{item:<Tx,x>-is-totbdd} \FI\ \eqref{item:TT-is-compact}.
  Let $\W$ be totally bounded. We show that $\T$, as an operator, is \wcbs.
  Given $x\in \H$, define $\phi_x(T)=\ip{Tx,x}$, for all $T\in \A$. Then
  $\phi_x(\T) = \ip{\T x,x}$ so that $\W = \set{\phi_x(\T):\|x\|=1}$.
  Since $\W$ is totally bounded in $\ell^p$, by Theorem \ref{thm:Kolmogorov-Riesz-compactness},
  for every $\e>0$, there is $m\in \N$ such that
  \begin{equation}\label{eqn:phi(a)-is-compact}
    \sum_{i>m}|\ip{T_ix,x}|^p = \sum_{i>m} |\phi_x(T_i)|^p < \e^p \quad (x\in \H, \|x\|=1).
  \end{equation}

  For every $\phi\in \A^*$, let
  \[
    V_\phi = \bigset{\phi'\in \A^*: |\phi'(T_i)-\phi(T_i)|< \frac\e{m^{1/p}}, 1\leq i \leq m}.
  \]
  Each $V_\phi$ is a neighbourhood of $\phi$ in the weak*\,topology. Since $\A^*_1$ is
  weak*\,compact, there exist $\phi_1,\dotsc,\phi_k$ in $\A^*_1$ such that
    $\A^*_1 \subset V_{\phi_1} \cup \dotsb \cup V_{\phi_k}$. Let
   \begin{equation*}
      \bx_j = \bigprn{\phi_j(T_1),\dotsc,\phi_j(T_m),0,0,0,\dotsc}
      \quad (1\leq j \leq k)
   \end{equation*}
  The finite set $\set{\bx_1,\dotsc,\bx_k}$ in $\ell^p$ defines
  a weak*\,neighbourhood of $0$ in $\ell^q$;
  \begin{align*}
    U  & = \set{\beta \in \ell^q: |\hat\be(\bx_j)|<\e, 1\leq j \leq k}\\
       & = \Bigset{\beta \in \ell^q: \Bigabs{\sum_{i=1}^m \beta_i \phi_j(T_i)}<\e, 1\leq j \leq k}.
  \end{align*}

  \noindent\textbf{Claim.} If $\be\in U$ and $\|\be\|_q\leq 1$ then $\|\T\be\| \leq 6\e$.\\
  \textit{Proof of claim}.
  If $\|x\|=1$ then $\phi_x\in \A^*_1$ and $\phi_x\in V_{\phi_j}$,
  for some $j\in\set{1,\dotsc,k}$. Then
  \begin{align*}
    |\hat \be\ip{\T x,x}|
      & = \Bigabs{\sum_{i=1}^\infty \beta_i \ip{T_ix,x}} = \Bigabs{\sum_{i=1}^\infty \beta_i \phi_x(T_i)}
        \leq \Bigabs{\sum_{i=1}^m \beta_i \phi_x(T_i)} + \Bigabs{\sum_{i>m} \beta_i \phi_x(T_i)}\\
      & \leq \biggabs{\sum_{i=1}^m \beta_i\bigprn{\phi_x(T_i)-\phi_j(T_i)}} + \Bigabs{\sum_{i=1}^m \beta_i \phi_j(T_i)}
              + \|\beta\|_q \Bigprn{\sum_{i>m} |\phi_x(T_i)|^p}^{1/p} \\
      & \leq \|\beta\|_q \Bigprn{\sum_{i=1}^m |\phi_x(T_i)-\phi_j(T_i)|^p}^{1/p} + \e + \e \\
      & \leq \biggprn{\sum_{i=1}^m \frac{\e^p}m}^{1/p} + \e + \e = 3\e.
  \end{align*}

  Hence, $|\hat\be \ip{\T x,x}| \leq \|x\|^2 (3\e)$, for all $x\in \H$.
  It is easily verified that
  \begin{equation*}
    \ip{\T x,y} = \frac14 \sum_{k=0}^3 \I^k \ip{\T (x+\I^k y),x+\I^k y} \quad (x,y\in \H).
  \end{equation*}
  If $\|x\|=\|y\|=1$ then
  \begin{align*}
    |\hat \be\ip{\T x,y}|
      & \leq \frac14 \sum_{k=0}^3 \bigabs{\hat \be \ip{\T (x+\I^k y),x+\I^k y}} \\
      & \leq \frac14 \sum_{k=0}^3 \|x+\I^k y\|^2 (3\e)
        \leq 2(3\e) = 6 \e.
  \end{align*}
  Taking supremum over all $x,y\in \H$ with $\|x\|=\|y\|=1$, using \eqref{eqn:norm(Tbe)}, we get
  \begin{equation*}
     \|\T\be\| \leq 6\e \quad (\be\in U, \|\be\|_q\leq 1).
  \end{equation*}
  This proves the claim. We conclude that $\T$ is \wcbs.
\end{proof}

\subsection{Relations to the numerical range}
The numerical range of a single operator $T\in\A$ is defined by
$W(T)=\set{\ip{Tx,x}:\|x\|=1}$, and the numerical radius of $T$ is given by
$w(T)=\sup\set{|\al|:\al\in W(T)}$. It is proved that
\begin{equation}\label{eqn:num-rad-ineq}
    \frac12 \|T\| \leq  w(T) \leq \|T\|.
\end{equation}

If $\T=(T_1,\dotsc,T_n)$ is an $n$-tuple of operators, the joint numerical range and the joint numerical radius
of $\T$ are defined, respectively, by
\begin{equation}\label{eqn:W(T)}
  \W(\T) = \set{\ip{\T x,x}:\|x\|=1},\quad
  \w(\T)=\sup\set{\|u\|:u\in \W(\T)},
\end{equation}
where $\ip{\T x,x} = \bigprn{\ip{T_1x,x},\dotsc,\ip{T_nx,x}}$.
In the definition of $\w(\T)$, some authors utilize the Euclidean norm $\enorm_2$,
while many researchers examine the more general case represented by the $p$-norm,
$p\geq 1$. We believe that the norm considered in defining $\w(\T)$ should essentially
stem from the norm associated with the $n$-tuple $\T$ (i.e.,\ the operator norm of $\T$).
For a background on (joint) numerical ranges and numerical radius inequalities,
see \cite{Dragomir-Mosleh-lect-num-rad-ineq-2022}.

If $\T=(T_1,T_2,\dotsc)$  is an infinite sequence of operators, it is natural to define
the joint numerical range $\W(\T)$ and the joint numerical radius $\w(\T)$ in the same way,
as in \eqref{eqn:W(T)}. However, in the infinite case, as Theorem \ref{thm:TT-is-compact}
indicates, the sequence $\T$ must satisfy certain conditions to ensure that $\W(\T)$ is a bounded
subset of some $\ell^p$-space.

We conclude the paper with a restatement of Theorem \ref{thm:TT-is-compact}.

\begin{thm}\label{thm:restatement}
  Let $1<p<\infty$, $1/p+1/q=1$. An operator $\T:\ell^q\to \A$ is compact if and only if
  the joint numerical range $\W(\T)$ is a totally bounded subset of $\ell^p$, in which case
  \begin{equation}\label{eqn:w(T)=sup(w(beT))}
    \frac12\|\T\| \leq \w(\T) = \sup\set{w(\T\beta): \|\beta\|_q \leq 1} \leq \|\T\|.
  \end{equation}
\end{thm}

\begin{proof}
  In light of Theorem \ref{thm:TT-is-compact}, it remains only to prove \eqref{eqn:w(T)=sup(w(beT))}.
  \begin{align*}
    \w(\T) & = \sup\set{\|\ip{\T x,x}\|: \|x\|=1} \\
           & = \sup\set{|\hat \be \ip{\T x,x}|: \|x\|=1, \|\be\|_q\leq 1} \\
           & = \sup\set{|\ip{(\T\be) x,x}|: \|x\|=1, \|\be\|_q\leq 1} \\
           & = \sup\set{w(\T \be): \|\be\|_q\leq 1}.
  \end{align*}

  For every $\be\in\ell^q$, by \eqref{eqn:num-rad-ineq}, we have
  $\frac12\|\T\be\| \leq \w(\T\be) \leq \|\T\be\|$. Taking supremum over all
  $\be$ with $\|\be\|_q\leq 1$, we get
  \begin{equation*}
    \frac12\|\T\| \leq \w(\T) \leq \|\T\|.
    \qedhere
  \end{equation*}
\end{proof}

It worth noting that
\begin{equation*}
  \w(\T) = \sup_{\|x\|=1}\Bigprn{\sum_{i=1}^\infty |\ip{T_ix,x}|^p}^{1/p}.
\end{equation*}

\subsection*{Declarations}
\begin{description}
  \item[Funding] The author declares that no funds, grants, or other support were received during
                 the preparation of this manuscript.

  \item[Competing Interests] The author has no relevant financial or non-financial interests to disclose.

  \item[Data availability] The manuscript has no associated data.
\end{description}

\subsection*{Acknowledgement}
The author expresses his sincere gratitude to the referees for their careful reading and suggestions
that improved the presentation of this paper.

%\printbibliography
%\end{document}


\begin{thebibliography}{9}

\bibitem{Abtahi-RM} 
 M. Abtahi and S. Farhangi, 
 \textit{A characterization of polynomially convex sets in Banach spaces.} 
 Results Math. \textbf{72} (2017), 2013--2021. 

\bibitem{Bottazzi-approx-com-oper-2013}
 T. Bottazzi and A. Varela,
 \textit{Best approximation by diagonal compact operators.}
 Linear Algebra Appl. \textbf{439} (2013), 3044--3056.


\bibitem{Dragomir-Mosleh-lect-num-rad-ineq-2022}
P. Bhunia, S. S. Dragomir, M. S. Moslehian, and K. Paul,
\textit{Lectures on numerical radius inequalities.}
Infosys Science Foundation Series in Mathematical Sciences,
Springer, Cham, 2022.

\bibitem{Kolmogorov-compactness-thm}
H. Hanche-Olsen and H. Holden,
\textit{The {K}olmogorov-{R}iesz compactness theorem.}
Expo. Math. \textbf{28} (2010), 385--394.

\bibitem{Horvath}
J. Horv\'ath, \textit{Topological vector spaces and distributions.} {V}ol. {I}.
Addison-Wesley Publishing Co., Reading, Mass.-London-Don Mills, Ont., 1966.

\bibitem{Mal-Sain-geo-prop-oper-2019}
A. Mal, D. Sain, and K. Paul,
\textit{On some geometric properties of operator spaces.}
Banach J. Math. Anal. \textbf{13} (2019), 174--191.


\bibitem{Rudin-FA}
W. Rudin, \textit{Functional analysis.} 2nd Edition,
International Series in Pure and Applied Mathematics,
McGraw-Hill, Inc., New York, 1991.

\bibitem{Sain-approx-comp-oper-2021}
D. Sain, \textit{On best approximations to compact operators.}
Proc. Amer. Math. Soc. \textbf{149} (2021), 4273--4286.

\bibitem{Wojcik-Ortho-compact-oper-2017}
P. W\'ojcik, \textit{Orthogonality of compact operators.}
Expo. Math. \textbf{35} (2017), 86--94.

\end{thebibliography}
\end{document}